\documentclass{article}

\usepackage{amsfonts}

\usepackage{amsthm}

\usepackage{amssymb}

\usepackage{amsmath}

\usepackage{enumerate}

\usepackage[all]{xy}

\usepackage{fancyhdr}

\usepackage{graphicx}

\long\def\symbolfootnote[#1]#2{\begingroup\def\thefootnote{\fnsymbol{footnote}}
\footnote[#1]{#2}\endgroup}

\newcommand{\N}{\mathbb{N}}

\newcommand{\R}{\mathbb{R}}

\newcommand{\C}{\mathbb{C}}

\newcommand{\Manoa}{M\=anoa}

\newcommand{\Hawaii}{Hawai\kern.05em`\kern.05em\relax i}

\newcommand{\supp}{\text{supp}}

\newcommand{\diam}{\text{diam}}

\theoremstyle{plain} \newtheorem*{question}{Questions}

\theoremstyle{remark} \newtheorem*{bgrem}{Remark}

\theoremstyle{definition} \newtheorem{adef}{Definition}[section]

\theoremstyle{definition} \newtheorem{ula}[adef]{Definition}

\theoremstyle{plain} \newtheorem{blem}[adef]{Lemma}

\theoremstyle{plain} \newtheorem{bmulem}[adef]{Lemma}

\theoremstyle{plain} \newtheorem{bbmulem}[adef]{Lemma}

\theoremstyle{definition} \newtheorem{bdef}[adef]{Definition}

\theoremstyle{plain} \newtheorem{bci}[adef]{Lemma}

\theoremstyle{definition} \newtheorem{propa}{Definition}[section]

\theoremstyle{plain} \newtheorem{ab}[propa]{Proposition}

\theoremstyle{definition} \newtheorem{cedef}[propa]{Definition}

\theoremstyle{plain} \newtheorem{osthe}[propa]{Theorem}

\theoremstyle{plain} \newtheorem{ceb}[propa]{Corollary}

\theoremstyle{definition} \newtheorem{mspdef}[propa]{Definition}

\theoremstyle{remark} \newtheorem{msprem}[propa]{Remark}

\theoremstyle{plain} \newtheorem{mspb}[propa]{Theorem}

\theoremstyle{definition} \newtheorem{onldef}[propa]{Definition}

\theoremstyle{remark} \newtheorem{onlrem}[propa]{Remark}

\theoremstyle{plain} \newtheorem{salem}[propa]{Lemma}

\theoremstyle{plain} \newtheorem{onlb}[propa]{Theorem}

\theoremstyle{definition} \newtheorem{sogdef}{Definition}[section]

\theoremstyle{plain} \newtheorem{expthe}[sogdef]{Theorem}

\theoremstyle{plain} \newtheorem{cocom}[sogdef]{Corollary}

\theoremstyle{definition} \newtheorem{boxdef}[sogdef]{Definition}

\theoremstyle{plain} \newtheorem{boxthe}[sogdef]{Theorem}

\theoremstyle{remark} \newtheorem{cubesrem}[sogdef]{Remark}

\theoremstyle{definition} \newtheorem{cg}{Definition}[section]

\theoremstyle{definition} \newtheorem{equiva}[cg]{Definition}

\theoremstyle{definition} \newtheorem{quantver}[cg]{Definition}

\theoremstyle{definition} \newtheorem{quantamen}[cg]{Definition}

\theoremstyle{definition} \newtheorem{fad}[cg]{Definition}

\theoremstyle{plain} \newtheorem{wmspamen}[cg]{Lemma}

\theoremstyle{plain} \newtheorem{amenthe}[cg]{Theorem}

\theoremstyle{plain} \newtheorem{nocor}[cg]{Corollary}

\theoremstyle{remark} \newtheorem{cerem}[cg]{Remark}

\author{Jacek Brodzki, Graham Niblo, J\'{a}n \v{S}pakula, Rufus Willett and \\ Nick Wright}

\title{Uniform Local Amenability}

\begin{document}

\maketitle

\begin{abstract}
The main results of this paper show that various coarse (`large scale') geometric properties are closely related.  In particular, we show that property A implies the operator norm localisation property, and thus that norms of operators associated to a very large class of metric spaces can be effectively estimated.

The main tool is a new property called uniform local amenability.  This property is easy to negate, which we use to study some `bad' spaces.  We also generalise and reprove a theorem of Nowak relating amenability and asymptotic dimension in the quantitative setting.
\end{abstract}

\symbolfootnote[0]{The first, second and fifth authors were supported by EPSRC grant EP/F031947/1.  The second author was supported by a Leverhulme Research Fellowship.  The third author was supported by the Deutsche
Forschungsgemeinschaft (SFB 878).  The fourth author was supported by the U.S. National Science Foundation.}

\section{Introduction}

The fundamental difficulty of computing operator norms arises in many areas of functional analysis.  In the context of metric geometry, it arises in the study of Roe algebras of finite propagation operators and is important in higher index theory \cite{Roe:1996dn} and theoretical physics \cite{Georgescu:2011cr}, among other places.  In higher index theory, the most important examples of metric spaces are often discrete groups.

The operator norm localisation property ($ONL$), which was introduced by Chen--Tessera--Wang--Yu in \cite{Chen:2008so}, provides a powerful tool for localising the problem.  It has been used to compute trace invariants and for other purposes in work on operator algebras and Baum-Connes type conjectures \cite{Guentner:2008gd,Willett:2010ud,Willett:2010zh,Spakula:2011bs}.

Here we show that $ONL$ holds for any bounded geometry metric space, and in particular any countable discrete group, which satisfies Yu's property A.  It is therefore more or less universally held: very few metric spaces are known which do not satisfy property A.  This generalises results of \cite{Chen:2008so,Chen:2009fk,Guentner:2008gd}, and combined with \cite{Dadarlat:2007qy}, it also reproves results from \cite{Chen:2008uq}.

The main technical tool used in the proof is the introduction of a property $ULA$ (`uniform local amenability') and its measure theoretic counterpart $ULA_\mu$; these should be viewed as local versions of Block-Weinberger amenability \cite{Block:1992qp}.  The additional properties of metric sparsification ($MSP$) and coarse embeddability ($CE$) have also been  studied in this context and we take the opportunity to record the relationships between these properties that were known to us when this paper was completed \eqref{diag}; these relationships are mainly new, and are of interest in their own right.

\begin{equation}\label{diag}
\xymatrix{  & CE & & ULA  \\ A \ar@{=>}[r]_{1~~~} & ULA_\mu \ar@{=>}[u]^{3} \ar@{=>}[urr]^{2} \ar@{<=>}_{4}[r] & MSP \ar@{=>}[r]_{5} & ONL \ar@{=>}[u]^{6} }.
\end{equation}
\begin{enumerate}
\item Proposition \ref{ab} below.  
\item Lemma \ref{bbmulem} below. The converse is open.
\item Corollary \ref{ceb} below.  The converse is false: see Corollary \ref{cocom}.
\item Theorem \ref{mspb} below.
\item This is proved in \cite[Section 4]{Chen:2008so}.  
\item Theorem \ref{onlb} below.  The converse is open.
\end{enumerate}

Of course, 2 follows from 4, 5 and 6, but we record it separately for the sake of the Remark at the end of this introduction.
Note in particular that this implies that all amenable groups have $MSP$ and $ONL$, a well-known open problem in the theory.  We also prove that for a box space, the properties $A$, $MSP$, $ULA$, $ULA_\mu$ and $ONL$ are equivalent.

Since this paper was completed, we learnt of a beautiful result of Hiroki Sako \cite{Sako:2012fk}, that $A$ and $ONL$ are in fact equivalent in general (for bounded geometry metric spaces).  It follows from this and our results that in fact the properties $A$, $MSP$, $ULA_\mu$, $ONL$ appearing in figure \eqref{diag} are all equivalent; combining our results with Sako's, we get the following quite satisfactory picture

$$
\xymatrix{  & CE & & ULA \\ A \ar@{<=>}[r] & ULA_\mu \ar@{=>}[u] \ar@{<=>}[r] & MSP \ar@{<=>}[r] & ONL  \ar@{=>}[u] }.
$$
Note that this gives an affirmative answer to Question 1 from \cite[page 1510]{Chen:2008so}.  Sako's methods are operator algebraic, using analysis of uniform Roe algebras, whereas ours are more purely coarse geometric; in particular, the two proofs are quite different.  It might be interesting to give a more coarse geometric proof of the implication $ONL\Rightarrow A$ from Sako's paper.

We leave the following questions open.

\begin{question}\label{question}
Are the properties $ONL$, $ULA$, $A$, $ULA_\mu$, $MSP$ all equivalent?  Does $ULA$ imply $CE$?\footnote{It follows from Corollary \ref{cocom} below that $CE$ cannot imply $ULA$.}
\end{question} 

Apart from being useful to prove the implications above, $ULA$ and $ULA_\mu$ have a significant advantage over the other properties in figure \eqref{diag}: it seems to be easier to check that they fail.  This allows us to give simple proofs that expanders and sequences of graphs with `large' girth do not have $ULA$.  In particular, this gives new examples of spaces without $ONL$, and reproves and generalises the main results of \cite{Willett:2010kx} and \cite{Khukhro:2011tw}.  Note also that it follows from our results, Sako's theorem, and an example of Arzhantseva--Guentner--\v{S}pakula \cite{Arzhantseva:2011vn} that $CE$ is strictly weaker than all the other properties in figure \eqref{diag}, apart possibly from $ULA$.  The relative ease with which $ULA$ and $ULA_\mu$ can be falsified may also play an important role in the construction of new non-exact groups (and more generally, metric spaces without property A), a task which to date has proved very difficult.

Finally, we look briefly at the quantitative aspects of the theory, using our ideas to give a new and more general proof of a theorem of Nowak \cite[Theorem 6.1]{Nowak:2007vn} relating quantitative versions of asymptotic dimension to quantitative versions of amenability.  There seems to be more that can be said here: in particular, we sketch an idea for constructing more examples of spaces with $CE$ but not $A$.

\begin{bgrem}\label{bgrem}
Throughout this piece, we make a blanket assumption that all metric spaces are of bounded geometry, and the implications above are in general only known to be valid under this assumption.  As some readers may be interested, we record whether the known proof of each implication in figure \eqref{diag} requires bounded geometry (where in each case, $ULA_\mu$ and $ULA$ are to be understood in the `set' definition rather than the `function' definition -- see Lemmas \ref{bmulem} and \ref{blem} below) : (1) yes; (2) no; (3) no; (4) no; (5) no; (6) yes.  We note that Sako's proof that $ONL$ is equivalent to $A$ uses bounded geometry in both directions.
\end{bgrem}

\subsection*{Outline of the paper}

In Section \ref{abbsec} below, we define $ULA$ and $ULA_\mu$, and discuss some basic properties. In Section \ref{mspsec} we recall the definitions of $A$, $CE$, $MSP$ and $ONL$, and fill in all the new implications in figure \eqref{diag}.  

The last two sections explore these properties.  In section \ref{countersec} we prove that expanders and sequences of graphs with large girth do not have property $ULA$, and that all the properties $A$, $ULA_\mu$, $ULA$, $MSP$, $ONL$ are equivalent for a box space.  Finally, in Section \ref{quantsec} we briefly discuss quantitative versions of our `local amenability' properties; as mentioned above, the main result is a generalisation of a theorem of Nowak.

\subsection*{Notation and conventions}

  If $X$ is a metric space, $x\in X$ and $E\subseteq X$, we use the following conventions.
\begin{align*}
B(x;R) & :=\{y\in X~|~d(x,y)\leq R\} \\
E^c & :=\{y\in X~|~y\not\in E\} \\
N_R(E) & :=\{y\in X~|~d(y,E)\leq R\} \\
\partial_RE &:= N_R(E) \backslash E.
\end{align*}

A metric space $X$ is said to be \emph{bounded geometry} if for all $R>0$ there exists $N_R\in\N$ such that $|B(x;R)|\leq N_R$.  Throughout, we say `$X$ is a space' as shorthand for `$X$ is a bounded geometry metric space'. Note that almost everything in this piece would go through if we only worked with metric spaces that are coarsely equivalent to some bounded geometry metric space (thus for example many manifolds).

A map between metric spaces $f:X\to Y$ is called a \emph{coarse embedding} if there exist non-decreasing functions $\rho_+,\rho_-:[0,\infty)\to[0,\infty)$ such that $\rho_-(t)\to \infty$ as $t\to\infty$ and 
$$
\rho_-(d_X(x_1,x_2))\leq d_Y(f(x_1),f(x_2))\leq \rho_+(d_X(x_1,x_2))
$$
for all $x_1,x_2$ in $X$.  $\rho_-$ and $\rho_+$ are called \emph{distortion functions} associated to $f$.  $f$ is called a \emph{coarse equivalence} if it is a coarse embedding, and if in addition there exists $C\geq 0$ such that for all $y\in Y$, $d(f(X),y)\leq C$.

\section{Uniform local amenability}\label{abbsec}

In this section we introduce our versions of `uniform local amenability', $ULA$ and $ULA_\mu$ below.  in order to motivate the definitions we recall the Block--Weinberger definition of amenability for a metric space \cite[Section 3]{Block:1992qp}.

\begin{adef}\label{adef}
A space $X$ is called \emph{amenable} if for all $R,\epsilon>0$ there exists a finite subset $E$ of $X$ (called a \emph{F\o lner set}) such that
$$
|\partial_R E|<\epsilon|E|.
$$
\end{adef}

Although this property (and its negation) are very useful in some contexts, it has shortcomings:  one of these is that it does not pass to subspaces.  As an attempt to rectify this, the following definition is very natural.

\begin{ula}\label{ula}
Let $X$ be a space.  $X$ is said to be \emph{uniformly locally amenable} if for all $R,\epsilon>0$ there exists $S>0$ such that for any finite subset $F$ of $X$ there exists $E\subseteq X$ such that $\diam(E)\leq S$ and 
$$
|\partial_R E\cap F|<\epsilon |E\cap F|.
$$
\end{ula}

Essentially, the definition says that all finite subsets of $X$ must be amenable, in such a way that amenability is seen by F\o lner sets of \emph{uniform} size.  Any finite metric space is of course trivially amenable in its own right: the non-triviality comes from requiring uniformity.  

The property in Definition \ref{ula} is equivalent to $ULA$ from Definition \ref{bdef} below; that definition is more convenient, but requires some preliminary lemmas. We will also need a stronger\footnote{A priori anyway: the two properties could be equivalent.} version of uniform local amenability, where probability measures rather than finite subsets are used to `localize'.  This is $ULA_\mu$, which is also introduced below.

The next two technical lemmas provide groundwork for these definitions.

\begin{bmulem}\label{bmulem}
Let $X$ be a space.  The following two properties are equivalent.
\begin{enumerate}
\item For all $R,\epsilon>0$ there exists $S>0$ such that for all probability measures $\mu$ on $X$ there exists a function $\phi\in l^1(X)$ such that
\begin{itemize}
\item $\text{diam}(\text{supp}(\phi))\leq S$;
\item the following `variational inequality' holds
\begin{equation}\label{var}
\sum_{x\in \supp(\mu)}\mu(x)\sum_{\substack{y\in \supp(\mu)\\ d(x,y)\leq R}}|\phi(x)-\phi(y)|<\epsilon \sum_{x\in \supp(\mu)} \mu(x)|\phi(x)|.
\end{equation}
\end{itemize}
\item For all $R,\epsilon>0$ there exists $S>0$ such that for all probability measures $\mu$ on $X$ there exists a finite set $E\subseteq X$ such that
\begin{itemize}
\item $\text{diam}(E)\leq S$;
\item the following inequality holds
$$
\mu(\partial_R E)<\epsilon \mu(E).
$$
\end{itemize}
\end{enumerate}
\end{bmulem}

The proof below is a standard argument based on that showing the `Reiter' formulation of amenability implies the `F\o lner' formulation: for the readers' convenience, we provide the details.

\begin{proof}
Assume the first condition holds.  Let $R,\epsilon>0$ be given, and let $\mu$ be a probability measure; we must find $S>0$ (which is independent of $\mu$) and $E$ as in the second condition.  Let $S$ and $\phi$ be as in the first condition; by replacing $\phi$ with $|\phi|$, we may assume that $\phi$ is non-negative.  Let $F_1=\supp(\phi)$ and let $F_1\supseteq F_2\supseteq \cdots \supseteq F_n$ be a sequence of (finite) subsets of $X$ such that we can write
$$
\phi=\sum_{i=1}^n \frac{a_i}{|F_i|}\chi_{F_i},
$$   
where each $a_i$ is a positive real number and $\chi_{F_i}$ is the characteristic function of $F_i$.   The variation inequality in line \eqref{var} above can then be rewritten
\begin{align*}
\sum_{i=1}^n\frac{a_i}{|F_i|} & \sum_{x\in \supp(\mu)}\mu(x)\sum_{\substack{y\in \supp(\mu)\\ d(x,y)\leq R}}|\chi_{F_i}(x)-\chi_{F_i}(y)| \\ & <\epsilon \sum_{i=1}^n\frac{a_i}{|F_i|}\sum_{x\in \supp(\mu)} \mu(x)\chi_{F_i}(x),
\end{align*}
from which it follows that for some $i$,
\begin{align*}
\sum_{x\in \supp(\mu)}& \mu(x)\sum_{\substack{y\in \supp(\mu)\\ d(x,y)\leq R}}|\chi_{F_i}(x)-\chi_{F_i}(y)| \\  &<\epsilon \sum_{x\in \supp(\mu)} \mu(x)\chi_{F_i}(x)=\epsilon \mu(F_i).
\end{align*}
On the other hand, the left-hand-side in the line above is at least $\mu(\partial_RF_i)$; it follows that $E:=F_i$ satisfies all the conditions in the second part of the Lemma.\\

Conversely, assume that the condition in the second part of the Lemma holds.  Let $R,\epsilon$, $\mu$ be as given.  Let $N_R$ be an absolute bound on the number of points in a ball of radius $R$, and let $S>0$ and $E'$ satisfy the conditions in the second part with respect to the data $2R,\epsilon/N_R,\mu$, so in particular
\begin{equation}\label{setineq}
\mu(\partial_{2R}E')<\frac{\epsilon}{N_R} \mu(E').
\end{equation}
Let $E=N_R(E')=E'\cup\partial_R(E')$.  Then $\partial_R(E)\cup \partial_R(X\setminus E)\subseteq \partial_{2R}(E')$ so 
\begin{equation}\label{setineq2}
\mu(\partial_R(E)\cup \partial_R(X\setminus E))\leq \mu( \partial_{2R}(E'))< \frac{\epsilon}{N_R}\mu(E')\leq \frac{\epsilon}{N_R}\mu(E).
\end{equation}

Let now $\phi$ be the characteristic function of $E$.  We have
\begin{align*}
& \sum\limits_{x\in\supp(\mu)} \mu(x) \sum\limits_{\substack{y\in \supp(\mu) \\ d(x,y)\leq R}} |\phi(x)-\phi(y) | \\ &=
\sum\limits_{x\in E} \mu(x) \sum\limits_{\substack{y\in \partial_R(E) \\ d(x,y)\leq R}} |\phi(x)-\phi(y)| +\sum\limits_{x\in X\setminus E} \mu(x) \sum\limits_{\substack{y\in \partial_R(X\setminus E)  \\ d(x,y)\leq R}} |\phi(x)-\phi(y)|\\ & =
\sum\limits_{x\in \partial_R(X\setminus E)} \mu(x) \sum\limits_{\substack{y\in \partial_R(E) \\ d(x,y)\leq R}} |\phi(x)-\phi(y) | +\sum\limits_{x\in \partial_R(E)} \mu(x) \sum\limits_{\substack{y\in \partial_R(X\setminus E) \\ d(x,y)\leq R}} |\phi(x)-\phi(y)|\\
& \leq N_R\sum\limits_{x\in \partial_R(X\setminus E)\cup \partial_R(E)}\mu(x)\\ &\leq N_R\mu( \partial_R(X\setminus E)\cup \partial_R(E));
\end{align*}
combining this with line \eqref{setineq2} gives
$$
\sum\limits_{x\in\supp(\mu)} \mu(x) \sum\limits_{\substack{y\in \supp(\mu) \\ d(x,y)\leq R}} |\phi(x)-\phi(y) |< \epsilon\mu(E)=\epsilon\sum_{x\in X}\mu(x)|\phi(x)|.
$$
The support of $\phi$ is just $E$, which is a subset of $\supp(\mu)$ of diameter at most $S+R$; we are done.
\end{proof}

The following Lemma can be proved in the same way as Lemma \ref{bmulem}.

\begin{blem}\label{blem}
Let $X$ be a space.  The following two properties are equivalent.
\begin{enumerate}
\item For all $R,\epsilon>0$ there exists $S>0$ such that for all finite subsets $F$ of $X$ there exists a function $\phi\in l^1(X)$ such that
\begin{itemize}
\item $\text{diam}(\text{supp}(\phi))\leq S$;
\item the following `variational inequality' holds
$$
\sum_{x\in F}\sum_{\substack{y\in F \\ d(x,y)\leq R}}|\phi(x)-\phi(y)|<\epsilon \sum_{x\in F}|\phi(x)|.
$$
\end{itemize}
\item For all $R,\epsilon>0$ there exists $S>0$ such that for all finite subsets $F$ of $X$ there exists a finite set $E\subseteq X$ such that
\begin{itemize}
\item $\text{diam}(E)\leq S$;
\item the following inequality holds
$$
|\partial_R E\cap F|<\epsilon |E\cap F|. \eqno \qed
$$
\end{itemize}
\end{enumerate} 
\end{blem}

\begin{bdef}\label{bdef}
We call the property appearing in Lemma \ref{bmulem} above $ULA_\mu$, or \emph{uniform local amenability}.  We call the property appearing in Lemma \ref{blem} $ULA$.
\end{bdef}

\begin{bci}\label{bci}
Let $Y$ be a space with $ULA$ (respectively, $ULA_\mu$), and $X$ be a space such that there exists a coarse embedding $f:X\to Y$.   Then $X$ has $ULA$ (resp. $ULA_\mu$).

In particular, $ULA$ and $ULA_\mu$ are coarse invariants.
\end{bci}

\begin{proof}
Let $f:X\to Y$ be a coarse embedding between spaces and let $\rho_{\pm}:[0,\infty)\to[0,\infty)$ be associated distortion functions. We assume (as we may) that $\rho_-$ has the following property: for all $t_1<t_2$, if $\rho_-(t_2)>0$, then $\rho_-(t_1)<\rho_-(t_2)$. Furthermore, denote $D=\sup_{y\in Y}|f^{-1}(y)|$.

We are going to show that if $Y$ has $ULA$ (or $ULA_\mu$), then $X$ does as well. We shall use the second formulation from Lemmas \ref{blem} and \ref{bmulem}. We give the argument for $ULA$.

Given $R,\epsilon>0$, declare $S=\rho^{-1}_-(S')$, where $S'$ comes from $ULA$ for $Y$ with parameters $\rho_+(R)$ and $\epsilon/D$; we now check that this $S$ satisfies the requirements.  Given a finite $F\subseteq X$, let $E'\subseteq Y$ be the set whose existence is guaranteed by $ULA$ for $Y$ with respect to the subset $f(F)$ and parameters $\rho_+(R)$ and $\epsilon/D$, i.e.\ $E'$ satisfies $\diam(E')\leq S'$ and 
$$|\partial_{\rho_+(R)}E'\cap f(F)|< \frac{\epsilon}D|E'\cap f(F)|.$$  
Passing to $E'\cap f(F)$, we may assume that $E'\subseteq f(F)$.

Denote $E=f^{-1}(E')\cap F$. Observe that if $0<d(x,E)\leq R$ for some $x\in F$, then $0<d(f(x),E')\leq \rho_+(R)$. Moreover, there are at most $D$ points whose image under $f$ coincides with $f(x)$. Consequently, $$
|\partial_RE\cap F|\leq D|\partial_{\rho_+(R)}E'\cap f(F)|<\epsilon|E'|\leq \epsilon|E|=\epsilon |E\cap F|.
$$
Finally, note that $\diam(E)\leq \rho_-^{-1}(\diam(E'))\leq S$. This finishes the argument for property $ULA$.

The argument for $ULA_\mu$ can be done along the same lines, using push-forward measures.
\end{proof}

The following Lemma is immediate: indeed, one should think of $ULA$ as being the special case of $ULA_\mu$ where the probability measure must be a normalised characteristic function. 

\begin{bbmulem}\label{bbmulem}
$ULA_\mu$ implies $ULA$. \qed
\end{bbmulem}

\section{Relationship of uniform local amenability with other properties}\label{mspsec}

In this section we look at the relationships between the uniform local amenability properties we introduced in the last section, and some other coarse properties: property A, coarse embeddability, the metric sparsification property, and operator norm localisation.

\subsection*{Property A}

The following definition of property A is due to Higson--Roe \cite[Lemma 3.5]{Higson:2000dp}.  

\begin{propa}\label{propa}
Let $X$ be a space, and $\text{Prob}(X)$ denote the simplex of probability measures on $X$, considered as a subset of $l^1(X)$.
We say that $X$ has \emph{property A} if for all $R,\epsilon>0$ there exists $S>0$ such that there exists $\xi:X\to \text{Prob}(X)$, denoted $x\mapsto \xi_x$, such that
\begin{itemize}
\item for all $x,y\in X$ with $d(x,y)\leq R$, $\|\xi_x-\xi_y\|<\epsilon$;
\item for all $x\in X$, $\xi_x$ in supported in $B(x;S)$.
\end{itemize}
\end{propa}

\begin{ab}\label{ab}
Property $A$ implies $ULA_\mu$.
\end{ab}

\begin{proof}
Assume $X$ has $A$, and let $R,\epsilon>0$, and $\mu$ be a probability measure on $X$.   Write $F:=\text{supp}(\mu)$.  Let $N_R$ denote the maximal number of points in a ball of radius $R$ in $X$ (hence also in $F$).  As $A$ for $X$ implies that all subsets of $X$ have $A$ uniformly\footnote{i.e.\ for all $R,\epsilon>0$ there exists $S>0$ such that for any subset $E$ of $X$ there exists $\xi:E\to \text{Prob}(E)$ satisfying the conditions in Definition \ref{propa} for $R,\epsilon,S$ - this is folklore, following for example from the proof of \cite[Proposition 4.2]{Tu:2001bs}.}, there exists $S>0$ independently of $\mu$ and a function $\xi:F\to \text{Prob}(F)$ satisfying the conditions in Definition \ref{adef} for the parameters $R,\epsilon/N_R,S$.  Fixing $x\in F$ for the moment, we then have
$$
\sum_{y\in F,~d(x,y)\leq R}\|\xi_x-\xi_y\|<N_R\epsilon/N_R=\epsilon,
$$
whence (using the fact that $\mu$ and each $\xi_x$ is a probability measure)
$$
\sum_{x\in F}\mu(x)\sum_{y\in F,~d(x,y)\leq R}\|\xi_x-\xi_y\|<\epsilon\sum_{x\in F}\mu(x)\|\xi_x\|.
$$
Expanding the norms on both sides gives
\begin{align*}
\sum_{z\in F}\sum_{x\in F}\mu(x)\sum_{y\in F,~d(x,y)\leq R}|\xi_x(z)-\xi_y(z)|<\epsilon \sum_{z\in F}\sum_{x\in F}\mu(x)|\xi_x(z)|,
\end{align*}
whence there exists $z_0\in F$ such that
$$
\sum_{x\in F}\mu(x)\sum_{y\in F,~d(x,y)\leq R}|\xi_x(z_0)-\xi_y(z_0)|<\epsilon \sum_{x\in F}\mu(x)|\xi_x(z_0)|;
$$
defining $\phi:F\to\C$ by $\phi(x)=\xi_x(z_0)$ (and extending $\phi$ to all of $X$ by setting it to be zero outside of $F$), we are done.  
\end{proof}

\subsection*{Coarse embeddings into Hilbert space}

\begin{cedef}\label{cedef}
Let $X$ be a space.  We say that $X$ has property $CE$ if $X$ admits a coarse embedding into a Hilbert space.
\end{cedef}

The following theorem is an immediate consequence of a result of Ostrovskii \cite[Theorem 2]{Ostrovskiil:2009kl}.

\begin{osthe}\label{osthe}
Let $X$ be a space that does not have $CE$.  Then there exists $\epsilon>0$ and $R>0$ such that for all $S>0$ there exists a probability measure $\mu$ on $X$ such that for all $E\subseteq\supp(\mu)$ with $\diam(E)\leq S$ we have
$$
\mu(\partial_R E)\geq \epsilon \mu(E).
$$
\end{osthe}

\begin{proof}
Let $D$ be as in the conclusion of \cite[Theorem 1]{Ostrovskiil:2009kl}.  Let $R=8D+1$ and $\epsilon=1/4D$ (so in Ostrovskii's notation, $\epsilon=\phi(D,R)$).    Let $n$ be so large that $n-R/2>S$, and let $\nu_n$ and $F$ be as given in the conclusion of \cite[Theorem 2]{Ostrovskiil:2009kl}.  We may then take $\mu$ to be the restriction of $\nu_n$ to $F$, and it is easy to check that our conclusion follows from \cite[Theorem 2]{Ostrovskiil:2009kl}.
\end{proof}

Of course, the conclusion of Theorem \ref{osthe} simply \emph{is} the negation of property $ULA_\mu$, so the following corollary is immediate.

\begin{ceb}\label{ceb}
For a space $X$, $ULA_\mu$ implies $CE$. \qed
\end{ceb}

\subsection*{The metric sparsification property}

We recall the following definition.

\begin{mspdef}\label{mspdef}
Let $X$ be a space.  $X$ is said to have the \emph{metric sparsification property} ($MSP$) if there exists $c>0$ such that for all $R>0$ there exists $S>0$ such that for all probability measures $\mu$ on $X$ there exists a subset $\Omega\subseteq X$ equipped with a decomposition 
$$
\Omega=\sqcup_i\Omega_i
$$
such that
\begin{itemize}
\item $\mu(\Omega)\geq c$;
\item for all $i$, $\diam(\Omega_i)\leq S$;
\item for all $i\neq j$, $d(\Omega_i,\Omega_j)>R$.
\end{itemize}
\end{mspdef}

\begin{msprem}\label{msprem}
It is proved in \cite[Proposition 3.3]{Chen:2008so} that if $X$ has $MSP$ for some constant $c>0$, then it has it for any $c$ with $0<c<1$.
\end{msprem}

\begin{mspb}\label{mspb}
For a space $X$, $MSP$ is equivalent to $ULA_\mu$.
\end{mspb}

\begin{proof}
Assume that $X$ has $MSP$.  Let $R,\epsilon>0$, and let $\mu$ be a probability measure on $X$.  Let $c>1/(1+\epsilon)$; by Remark \ref{msprem}, we may assume that this is the `$c$' in the definition of $MSP$.  Let $S>0$ and $\Omega=\sqcup_i\Omega_i$ be a decomposition as in the definition of $MSP$ with respect to the parameter $2R$ and the probability measure $\mu$.  As the collection $\{N_R(\Omega_i)\}_i$ is disjoint, and as $\mu(\Omega)\geq c$, we have
$$
\sum_i \mu(N_R(\Omega_i))\leq \mu(X)=1\leq \frac{1}{c}\mu(\Omega)=\frac{1}{c}\sum_i\mu(\Omega_i).
$$
It follows that there exists $i$ such that if $E:=\Omega_i$ then
$$
\mu(N_R(E))\leq \frac{1}{c}\mu(E);
$$
as the left-hand-side is simply $\mu(E)+\mu(\partial_RE)$, however, this rearranges to
$$
\mu(\partial_RE)\leq \Big(\frac{1}{c}-1\Big)\mu(E),
$$
and by choice of $c$, this is the desired conclusion.\\

For the converse, assume $X$ has $ULA_\mu$.  Let $R>0$ and a probability measure $\mu_1:=\mu$ be given; by an approximation argument, we may assume that $F_1:=\supp(\mu_1)$ is a finite set.  Fix any $\epsilon>0$, and let $S>0$ be as given by $ULA_\mu$.  Let now $E_1\subseteq F_1$ have diameter at most $S$, and be such that 
$$
\mu(\partial_R E_1)<\epsilon \mu(E_1).
$$ 
Set $F_2:=F_1\backslash (N_R(E_1))$ and define $\mu_2$ to be the restriction of $\mu$ to $F_2$, renormalised so as to be a probability measure; by the definition of $ULA_\mu$ there exists $E_2\subseteq F_2$ such that
$$
\mu_2(\partial_RE_2)<\epsilon \mu_2(E_2).
$$
Note, however, that restricted to $F_2$, $\mu_2$ is just a rescaling of $\mu$, whence we also have
$$
\mu(\partial_R E_2\cap F_2)<\epsilon \mu(E_2).
$$

Similarly, we may now set $F_3=F_1\backslash (N_R(E_1)\cup N_R(E_2))$, and continue the process.  It must eventually terminate (as $F_1$ is finite) giving us sequences $F_1\supseteq F_2\supseteq\cdots \supseteq F_n$ and $E_1,...,E_n$ such that $E_i\subseteq F_i$ for all $n$ and so that:
\begin{itemize}
\item for all $i$, $\diam(E_i)\leq S$;
\item for $i\neq j$, $d(E_i,E_j)>R$;
\item for all $i$, $\mu(\partial_R E_i\cap F_i)<\epsilon \mu(E_i)$.
\end{itemize}
Set $\Omega_i:=E_i$ and $\Omega:=\sqcup \Omega_i$.  We have finally that
$$
1=\mu(F_1)=\sum_{i=1}^n \mu(E_i)+\mu(\partial_R E_i\cap F_i)<\sum_{i=1}^n(1+\epsilon)\mu(E_i)=(1+\epsilon)\mu(\Omega).
$$
Hence $\mu(\Omega)\geq 1/(1+\epsilon)$, so we may take $c=1/(1+\epsilon)$ and are done.
\end{proof}

\subsection*{The operator norm localisation property}

We give the following definition of the operator norm localisation property ($ONL$): it is easily seen to be equivalent to the original definition (\cite[Definition 2.2]{Chen:2008so}).

\begin{onldef}\label{onldef}
Let $X$ be a space and $\mu$ a positive measure on $X$.  Let $\mathcal{H}$ be a separable Hilbert space, and consider the space of functions $\phi:X\to\mathcal{H}$ such that the norm
$$
\|\phi\|^2:=\sum_{x\in X}\mu(x)\|\phi(x)\|^2_{\mathcal{H}}
$$
is finite.  Taking the quotient by the subspace of functions of norm zero\footnote{This does form a subspace by a version of Cauchy-Schwarz applied to the obvious inner product associated to $\|\cdot\|$.} gives a Hilbert space, which we denote $l^2(X,\mu,\mathcal{H})$.  Recall that any (bounded) operator $T$ on this Hilbert space can be considered as a matrix $T=(T_{x,y})_{x,y\in X}$, where each $T_{x,y}$ is a bounded operator on $\mathcal{H}$.  

For each $R>0$, define $\C_R[X;\mu,\mathcal{H}]$ to be the collection of bounded operators $(T_{x,y})$ on $l^2(X,\mu,\mathcal{H})$ such that $T_{x,y}=0$ whenever $d(x,y)>R$.

We say that $X$ has the \emph{operator norm localisation property} ($ONL$) if there exists $c>0$ such that for any $R>0$ there exists $S>0$ such that for any probability measure $\mu$ and separable $\mathcal{H}$, and any $T\in \C_R[X;\mu,\mathcal{H}]$ there exists a unit vector $\phi\in l^2(X,\mathcal{H})$ such that $\diam(\supp(\phi))\leq S$ and such that
$$
\|T\phi\|_{l^2(X,\mu,\mathcal{H})}\geq c\|T\|_{\mathcal{B}(l^2(X,\mu,\mathcal{H}))}.
$$
\end{onldef}

\begin{onlrem}\label{onlrem}
It is proved in \cite[Proposition 2.4]{Chen:2008so} that if $X$ has $ONL$ for some constant $c>0$, then it has it for any $c$ with $0<c<1$.
\end{onlrem}

We will need the following technical Lemma.
 
\begin{salem}\label{salem}
Let $X$ be a space with $ONL$.  Then there exists $c>0$ such that for all $R,M>0$ there exists $S>0$ such that for all positive measures $\mu$ and separable Hilbert spaces $\mathcal{H}$, and any positive operator $A\in \C_R[X;\mu,\mathcal{H}]$ with $\|A\|\leq M$ there exists norm one $\psi\in l^2(X,\mu,\mathcal{H})$ such that $\diam(\supp(\psi))\leq S$ and 
$$
\langle \psi,A\psi\rangle\geq c\|A\|.
$$
Moreover, one can take any $c$ with $0<c<1$.
\end{salem}

\begin{proof}
Let $R,M$ and $A$ be as in the statement.  
Let $q$ be the (positive) square root function on $[0,M]$, and let $(p_n)$ be a sequence of polynomials converging uniformly to $q$ on $[0,M]$.  Let $\epsilon>0$ and $p_n$ be such that $\|q-p_n\|<\epsilon$, and note that for any $R>0$ there exists $R'$ such that if $T\in \C_R[X;\mu,\mathcal{H}]$, then $p_n(T)\in \C_{R'}[X;\mu,\mathcal{H}]$.

Let now $c$ be as given in the definition of $ONL$ for $X$, and let $S$ be as in the definition of $ONL$ with respect to the parameter $R'$. Let $\psi\in l^2(X,\mu,\mathcal{H})$ be of norm one, with diameter of support at most $S$, and such that $\|p_n(A)\psi\|\geq c\|p_n(A)\|$.  Then 
\begin{align*}
\sqrt{\langle \psi,A\psi\rangle} & =\|q(A)\psi\|\geq \|p_n(A)\psi\|-\|q_n(A)-p_n(A)\| \\
& \geq c\|p_n(A)\|-\epsilon\geq c(\|q(A)\|-\epsilon)-\epsilon \\ &=c\sqrt{\|A\|}-\epsilon(1+c);
\end{align*}
as $\epsilon$ was arbitrary (of course, $S$ implicitly depends on $\epsilon$, but all that matters is that for each $\epsilon$, some $S$ exists), this completes the proof (one also alters $c$ slightly to get $c$ as in the statement).  

To see that one can get any $c$ with $0<c<1$, it suffices to use Remark \ref{onlrem}, and then again use that $\epsilon$ was arbitrary. 
\end{proof}

\begin{onlb}\label{onlb}
If a space $X$ has $ONL$, then it has $ULA$.
\end{onlb}

\begin{proof}
Let $X$ be a space with $ONL$.  Let $R,\epsilon>0$. Let $F$ be a finite subset of $X$, and take $\mu$ to be the measure given by the characteristic function of $F$.  Take $\mathcal{H}=\C$, so that $l^2(X,\mu,\mathcal{H})$ identifies naturally with $l^2(F)$; we make this identification without further comment, and denote by $\{\delta_x\}_{x\in F}$ the canonical basis of this Hilbert space.  

Now, let $\Delta_R$ be the operator on $l^2(F)$ defined by
$$
\Delta_R:\delta_x\mapsto \sum_{y\in F,~d(x,y)\leq R}\delta_x-\delta_y;
$$
note that if $N_R$ is an absolute bound for the cardinality of a ball of radius $R$ in $X$, then $\|\Delta_R\|\leq 2N_R$.  Let $A_R=\|\Delta_R\|-\Delta_R$, which is an operator in $\C_R[X;\mu,\mathcal{H}]$.  An explicit computation (cf. line \eqref{lap} below) shows that $\Delta_R$ is a positive operator, whence $A_R$ is too and $\|A_R\|=\|\Delta_R\|$.

Choose $c$ with
$$
1>c>1-\Big(\frac{\epsilon}{2\sqrt{N_R}}\Big)^2.
$$
Lemma \ref{salem} implies that there exists $S$ (depending on $R$ and $\epsilon$, but not on $F$) and a norm one function $\psi\in l^2(F)$ with $\diam(\supp(\psi))\leq S$ and such that 
$$
\langle \psi,A_R\psi\rangle\geq c\|A_R\|.
$$
Expanding and rearranging this gives that $\langle \psi,\Delta_R\psi\rangle \leq 1-c$, and further expanding the left-hand-side and rearranging gives
\begin{equation}\label{lap}
\langle \psi,\Delta_R\psi\rangle=\frac{1}{2}\sum_{x\in F}\sum_{\substack{y\in F \\ d(x,y)\leq R}}|\psi(x)-\psi(y)|^2\leq 1-c.
\end{equation}
On the other hand, Cauchy-Schwarz implies that if $\phi(x)=|\psi(x)|^2$ then
\begin{align*}
\sum_{x\in F}& \sum_{\substack{y\in F \\ d(x,y)\leq R}}|\phi(x)-\phi(y)| \\ & \leq \Big(\sum_{x\in F}\sum_{\substack{y\in F \\ d(x,y)\leq R}}|\psi(x)-\psi(y)|^2\Big)^\frac{1}{2}\Big(\sum_{x\in F}\sum_{\substack{y\in F \\ d(x,y)\leq R}}|\psi(x)+\psi(y)|^2\Big)^\frac{1}{2}.
\end{align*}
Moreover, if $N_R$ is the maximal number of points in a ball of radius $R$ in $X$, this and the fact that $\|\psi\|=1$ in turn imply that
$$
\sum_{x\in F}\sum_{\substack{y\in F\\ d(x,y)\leq R}}|\phi(x)-\phi(y)|\leq 2\sqrt{N_R}\Big(\sum_{x\in F}\sum_{\substack{y\in F\\ d(x,y)\leq R}}|\psi(x)-\psi(y)|^2\Big)^\frac{1}{2}.
$$
Finally, combining this with line \eqref{lap} above gives that
$$
\sum_{x\in F}\sum_{\substack{y\in F\\ d(x,y)\leq R}}|\phi(x)-\phi(y)|\leq 2\sqrt{N_R}\sqrt{1-c}=2\sqrt{N_R}\sqrt{1-c}\sum_{x\in F}|\phi(x)|,
$$
where the equality uses again that $\|\psi\|=1$ and the definition of $\phi$; the choice of $c$ completes the argument.
\end{proof}

\section{Examples: spaces of graphs}\label{countersec}

In this section, we give two examples of spaces without $ULA$: expanding graphs, and sequences of graphs with large girth.  The first of these generalises and reproves \cite[Theorem 6.5]{Chen:2008so} and the main theorem of \cite{Khukhro:2011tw}, while the second generalises and reproves the main result from \cite{Willett:2010kx}.  We also prove that for a box space of a discrete group, all the properties $ULA$, $ULA_\mu$, $A$, $MSP$ and $ONL$ are equivalent, and equivalent to amenability of the original group.  Finally, we discuss how the non-bounded geometry examples of Nowak with $CE$ but not $A$ \cite{Nowak:2007vn} fit into our framework.

We start by recalling the definitions.

\begin{sogdef}\label{sogdef}
A space $X$ is called a \emph{space of graphs} if there exists a sequence $(X_n)$ of finite connected graphs such that $X=\sqcup X_n$ (as a set), and if the metric on $X$ restricts to the edge metric on each $X_n$, and is such that $d(X_n,X_n^c)\to\infty$ as $n\to\infty$.   Note that as we assume $X$ has bounded geometry, all the vertex degrees of all the $X_n$ must be uniformly bounded.

A space of graphs $X$ is called an \emph{expander} if $|X_n|\to\infty$ as $n\to\infty$ and there exists $\epsilon>0$ such that whenever $A\subseteq X_n$ is such that $|A|\leq |X_n|/2$ we have that $|\partial_1A|\geq \epsilon|A|$.

A space of graphs $X$ is said to have \emph{large girth} if $\text{girth}(X_n)\to\infty$ as $n\to\infty$ (recall that the \emph{girth} of a finite graph is the length of its shortest non-trivial cycle). 
\end{sogdef}

\begin{expthe}\label{expthe}
Say $X$ is a space of graphs.  If $X$ is an expander, or if $X$ has large girth and all vertices have degree at least three, then $X$ does not have $ULA$.
\end{expthe}

\begin{proof}
Assume first that $X$ is an expander, let $S>0$, and let $N_S$ be the maximum number of vertices in a ball of radius $S$.  Let $n$ be so large that $|X_n|\geq 2N_S$, and set $F=X_n$.  Then if $E\subseteq X_n$ is such that $\diam(E)\leq S$ we have $|\partial_1E|\geq \epsilon |E|$ by the expander assumption.  This contradicts $ULA$.\\

Say now that $X$ has large girth, and all vertices in $X$ have degree at least three.  Let $D$ be an upper bound on the degrees of all vertices in $X$.  Let $S>0$, and let $n$ be so large that any subset of diameter $S+1$ of $X_n$ is isometric to a subset of a tree (with all vertices of degrees at least three).  Then \cite[Lemma 3.3]{Willett:2010kx} implies that for any subset $E$ of $X_n$ such that $\diam(E)\leq S$ we have $|\partial_1E|\geq \frac{1}{D-1}|E|$ (the definition of `$\partial_1$' used in \cite{Willett:2010kx} is different to that used here, but this does not matter); again, this contradicts $ULA$.
\end{proof}

\begin{cocom}\label{cocom}
None of the properties $A$, $ULA_\mu$, $ULA$, $MSP$ or $ONL$ are implied by $CE$.  In particular, $ULA_\mu$, $A$ and $MSP$ are all strictly stronger than $CE$.
\end{cocom}

\begin{proof}
Arzhantseva--Guentner--\v{S}pakula \cite{Arzhantseva:2011vn} have given an example of a space with $CE$, which is actually a space of graphs with large girth (see also Ostrovskii \cite{Ostrovskii:2011oq}).  The corollary is immediate from this, Theorem \ref{expthe}, and the results of section \ref{mspsec}.
\end{proof}

Combining our results with Sako's, it follows that in fact $CE$ is strictly stronger than all of the properties $ULA_\mu$, $MSP$, $A$, $ONL$.  

We now look at box spaces.

\begin{boxdef}\label{boxdef}
Let $\Gamma$ be an infinite finitely generated discrete group, and let $\Gamma_1\geq \Gamma_2\geq \Gamma_3\geq\cdots$ be a decreasing sequence of finite index normal subgroups of $\Gamma$ such that $\cap_n\Gamma_n=\{e\}$.  Fix a generating set of $\Gamma$, and for each $n$ let $X_n=\Gamma/\Gamma_n$, equipped with the (Cayley) graph structure coming from the fixed finite generating set of $\Gamma$.

The \emph{box space} associated to this data is the\footnote{It is unique up to coarse equivalence.} space of graphs associated to the sequence of graphs $(X_n)$. 
\end{boxdef}

\begin{boxthe}\label{boxthe}
Let $\Gamma$ be a finitely generated discrete group, and $X=\sqcup X_n$ an associated box space.  Then the following are equivalent:
\begin{enumerate}
\item $\Gamma$ is amenable;
\item $X$ has all of the properties $A$, $ULA$, $ULA_\mu$, $MSP$, $ONL$.
\item $X$ has one of the properties $A$, $ULA$, $ULA_\mu$, $MSP$, $ONL$.
\end{enumerate}
\end{boxthe}

\begin{proof}
It is a theorem of Guentner--Roe \cite[Proposition 11.39]{Roe:2003rw} that amenability of $\Gamma$ implies property $A$ for $X$ (and hence, looking back at diagram \eqref{diag}, all of the other properties).  To complete the proof of the theorem, diagram \eqref{diag} implies that it suffices to prove that if $X$ has property $ULA$, then $\Gamma$ is amenable. 

Assume then that $\Gamma$ has $ULA$ and let $\epsilon>0$.  Then there exists $S$ such that for all $n$ there exists $E_n\subseteq X_n$ with $\diam(E_n)\leq S$ and $|\partial_1E_n|<\epsilon|E_n|$.  For $n$ suitably large, $E_n$ lifts isometrically to a finite subset $\widetilde{E}$ of $\Gamma$ such that $\partial_1(\widetilde{E})$ is also an isometric lift of $\partial_1E_n$.  In particular, $|\partial_1\widetilde{E}|<\epsilon|\widetilde{E}|$, which implies that $\Gamma$ is amenable.
\end{proof}

Combining our results with Sako's gives the result that in fact the properties $A$, $ULA_\mu$, $MSP$, $ONL$ are all equivalent for bounded geometry metric spaces.  It would be interesting to know if $ULA$ was also equivalent to the others in general.

\begin{cubesrem}\label{cubesrem}
Recall that Nowak \cite{Nowak:2007vn} has constructed $CE$ metric spaces without $A$, although without bounded geometry. Roughly, these are coarse disjoint unions of powers\footnote{Endowed with $l^1$-metric.} of a fixed finite group.  Nowak's spaces are spaces of graphs in our sense if one drops the bounded geometry assumption.

Note that the implications $MSP\implies ULA_\mu \implies ULA$ work for discrete metric spaces, even without assuming bounded geometry (cf.\ Remark \ref{bgrem} above). Observe that Nowak's examples do not have $ULA$ (whence also not $ULA_\mu$ or $MSP$). Indeed, $ULA$ for a coarse disjoint union of finite spaces $X_n$ implies that the pieces are amenable in a uniform way, i.e.\ that for a given $\epsilon>0$ there exists a F\o lner set $E_n\subseteq X_n$ for each $n$ such that $\diam(E_n)$ is uniformly bounded.   However, \cite[Theorem 4.4]{Nowak:2007vn} implies that this is not possible for Nowak's examples.
\end{cubesrem}


\section{Comments on the quantitative theory, and a theorem of Nowak}\label{quantsec}

In this section, we make a few brief comments on quantitative versions of the theory, concentrating on the relationship between this and established invariants in the case of amenable groups.  This enables us to recover (and slightly generalise) a result of Nowak: \cite[Theorem 6.1]{Nowak:2007ys}.  The section is just meant to give a flavour of what is possible; one could no doubt say rather more.  

We also include an unrelated comment on using quantitative properties to tell the difference between $CE$ and $A$.  

Throughout, we will focus on coarsely geodesic spaces as in the next definition.  The definition is slightly restrictive, but covers the motivating examples: graphs and finitely generated discrete groups.  On the other hand, it is easy to see that any bounded geometry metric space that is `quasi-geodesic' in any reasonable sense if quasi-isometric to one of this form.  Equivalently, any monogenic coarse structure can be metrised with a metric of this form: see for example \cite[Proposition 2.57]{Roe:2003rw}.

\begin{cg}\label{cg}
Let $(X,d)$ be a space.  $X$ is said to be \emph{coarsely geodesic} if the metric $d$ is integer valued, and if for any $n\in\N$, $x,y\in X$ we have that $d(x,y)=n$ if and only if there exists a sequence 
$x=x_0,x_1,...,x_n=y$ such that $d(x_i,x_{i+1})=1$ for all $i=0,...,n-1$.
\end{cg}

In our context `quantitative' means `measured by a given function'.  We will work with functions up to the following notions of order and equivalence.

\begin{equiva}\label{equiva}
Let $f,g:\N\to \N$ be functions.  We write $f\preceq g$ if there exist constants $c,d$ such that for all $n\in\N$, $f(n)\leq cg(dn)$.  We say that $f$ and $g$ are \emph{equivalent}, and write $f\sim g$ if $f\preceq g$ and $g \preceq f$.
\end{equiva}

The following is perhaps the most general `quantitative version' of one of the properties that we have studied.  Although it makes sense in general, it only seems to have much content in the case of coarsely geodesic spaces as above.

\begin{quantver}\label{quantver}
Let $X$ be a space, $c>0$ and $f:\N\to\N$ be a non-decreasing function.  

$X$ is said to have $WMSP(c,f)$ if for all $R\in\N$ and all finite subsets $F$ of $X$ there exists a subset $\Omega\subseteq F$ equipped with a decomposition 
$$
\Omega=\sqcup_i\Omega_i
$$
such that
\begin{itemize}
\item $|\Omega|\geq c|F|$;
\item for all $i$, $\diam(\Omega_i)\leq f(R)$;
\item for all $i\neq j$, $d(\Omega_i,\Omega_j)>R$.
\end{itemize}
\end{quantver}

The same proof as in Theorem \ref{mspb} above shows that $WMSP(c,f)$ is equivalent to an appropriate quantitative version of property $ULA$.  The results of Section \ref{mspsec} then show that $WMSP(c,f)$ (which should be thought of as standing for `weak metric sparsification property with respect to $c,f$') is implied by appropriate quantitative versions of $A$, $ULA_\mu$, $MSP$ and $ONL$; we leave the details to the reader.

Note that if $X,Y$ are quasi-isometric metric spaces, and $X$ has $WMSP(c,f)$, then $Y$ has $WMSP(c',g)$ for some $c'>0$ and $g$ with $g\sim f$.  The argument is standard and not directly relevant, so we omit it.

We will need the following two quantitative properties.  Our aim is to relate them to $WMSP(c,f)$, and thus to each other.

\begin{quantamen}\label{quantamen}
Let $X$ be an amenable space.  The \emph{isodiametric function} of $X$, $A_X:\N\to \N$ is defined by
$$
A_X(n):=\min\{\text{diam}(A)~|~A\subseteq X,~~|\partial_1(A)|\leq (1/n)|A|\}.
$$
\end{quantamen}

\begin{fad}\label{fad}
Let $X$ be a space.  Let $n$ be a natural number, and $\tau:\N\to \N$ be a non-decreasing function.  $X$ is said to have \emph{asymptotic dimension at most $n$} (with respect to $\tau$), in brief $X$ has $FAD(n,\tau)$, if for all $R\in\N$ there exist subsets $\Omega^1,...,\Omega^{n+1}$ of $X$ and decompositions $\Omega^i=\sqcup_{j\in I_i}\Omega^i_j$ such that
\begin{enumerate}
\item $X=\cup_{i=1}^{n+1} \Omega^i$;
\item for each $i=1,...,n+1$ and all $j_1\neq j_2$, $d(\Omega^i_{j_1},\Omega^i_{j_2})>R$;
\item for each $i=1,...,n+1$ and all $j$, $\text{diam}(\Omega^i_j)\leq \tau(R)$.
\end{enumerate}
\end{fad}

\begin{wmspamen}\label{wmspamen}
Let $X$ be an amenable space in the sense of Definition \ref{adef} above, and assume also that $X$ has $WMSP(c,f)$.  Then for any $R\in\N$ there exists a finite (non-empty) subset $E\subseteq X$ such that $\diam(E)\leq f(2R)$ and
$$
|N_R(E)|\leq\frac{2}{c}|E|.
$$
\end{wmspamen}

\begin{proof}
As $X$ is amenable there exists a finite subset $F\subseteq X$ such that $|\partial_R F|< |F|$.  Let $\Omega\subseteq F$ be as given in Definition \ref{quantver} above with respect to the parameter $2R$; we may assume all the $\Omega_i$s are non-empty.  It follows from this that
\begin{align*}
\Big(\sum_i|N_R(\Omega_i)|\Big)-|\partial_RF| & \leq \sum_i (|N_R(\Omega_i)|-|N_R(\Omega_i)\cap \partial_R F|) \\ & =\sum_i |N_R(\Omega_i)\cap F|\leq |F|
\end{align*}
whence 
$$
\sum_i |N_R(\Omega_i)|\leq 2|F|
$$
by assumption on $\partial_RF$.  From this and the fact that $|\Omega|\geq c|F|$, we see that
\begin{equation}\label{basic}
\sum_i|N_R(\Omega_i)|\leq \frac{2}{c}\sum_i|\Omega_i|.
\end{equation}
Finally, there exists $i$ such that if $E:=\Omega_i$ then
$$
|N_R(E)|\leq \frac{2}{c}|E|
$$
and of course $\diam(E)\leq f(2R)$.
\end{proof}

\begin{amenthe}\label{amenthe}
Let $X$ be a coarsely geodesic amenable space with $WMSP(c,f)$.

Then $A_X\preceq f$.
\end{amenthe}

\begin{proof}
We assume $X$ is unbounded, otherwise the result is trivial.  Fix for the moment $R\in\N$.  Using Lemma \ref{wmspamen}, we see there exists a finite subset $E\subseteq X$ such that $|N_R(E)|\leq \frac{1}{c}|E|$.  We may rewrite this inequality as
$$
\frac{|N_R(E)|}{|N_{R-1}(E)|}\cdots \frac{|N_2(E)|}{|N_1(E)|}\frac{|N_1(E)|}{|E|}\leq \frac{1}{c},
$$ 
whence (using the coarsely geodesic property) there exists $m=0,...,R-1$ such that if $A_R=N_m(E)$, then 
$$
\frac{|N_1(A_R)|}{|A_R|}\leq \Big(\frac{1}{c}\Big)^\frac{1}{R}.
$$ 
Rearranging this slightly gives
\begin{equation}\label{Ainq}
\frac{|\partial_1(A_R)|}{|A_R|}\leq \Big(\frac{2}{c}\Big)^{\frac{1}{R}}-1
\end{equation}
and $\text{diam}(A_R)\leq R+f(2R)$.  Note moreover that because $X$ is unbounded and coarsely geodesic we have that $f(n)\geq n$ for all $n\in\N$; in particular, then,
\begin{equation}\label{finq}
\text{diam}(A_R)\leq 2f(2R).
\end{equation}
Now, we have
$$
\frac{1}{n}\leq \Big(\frac{2}{c}\Big)^{\frac{1}{R}}-1~~ \Leftrightarrow ~~R\leq \frac{\log(2/c)}{\log(\frac{n+1}{n})},
$$
whence lines \eqref{Ainq} and \eqref{finq} together imply that
$$
A_X(n)\leq 2f\Big(\frac{2\log(2/c)}{\log(\frac{n+1}{n})}\Big).
$$
Finally, note that 
$$
\frac{1}{\log(\frac{n+1}{n})}\leq \frac{n}{\log(2)}
$$
for all $n\in\N$ (treating, as we may, the left hand side as zero when $n=0$) and $f$ is non-decreasing, whence
$$
A_X(n)\leq 2f\Big(\frac{2\log(2/c)n}{\log(2)}\Big)\preceq f(n)
$$
as required.
\end{proof}

The following theorem immediately implies \cite[Theorem 6.1]{Nowak:2007ys}.

\begin{nocor}\label{nocor}
Let $X$ be a coarsely geodesic, amenable space with $FAD(n,\tau)$.  Then $A_X\preceq \tau$.
\end{nocor}

\begin{proof}
It is easy to see that $FAD(n,\tau)$ implies $WMSP(\frac{1}{n+1},\tau)$.  The corollary is thus immediate from Theorem \ref{amenthe}.
\end{proof}

We conclude this section with a remark on quantitative phenomena and coarse embeddings.

\begin{cerem}\label{cerem}
One may try to define a `quantitative negation' of $ULA_\mu$ as follows.

\vspace{.1in}
\noindent \emph{Let $X$ be a space and $f:\N\to\R_+$ a non-decreasing, non-zero function.  Then $X$ has $\neg ULA_\mu(f)$ if for all $R>0$ and all $S>0$ there exists a probability measure $\mu$ on $X$ such that for all subsets $E\subseteq \supp (\mu)$ with $\diam(E)\leq S$ we have
$$
\mu(\partial_R E)\geq f(R)\mu(E).
$$}
Now, inspection of the proof of Theorem \ref{osthe} above, and the ingredients for it from \cite{Ostrovskiil:2009kl}, shows that if $X$ does not have $CE$ then it has $\neg ULA_\mu(f)$ with $f$ \emph{growing at least linearly}.  On the other hand, for a space not to have $ULA_\mu$ it suffices to show that it has $\neg ULA_\mu(f)$  for any (non-zero, non-decreasing) $f$.  In particular, this gives a potential quantitative approach to finding more examples of spaces without $ULA_\mu$ (hence without $A$), but with $CE$.   

Note in this regard also that one easily sees that either of the classes of spaces in Theorem \ref{expthe} have $\neg ULA_\mu(f)$ with $f$ growing \emph{exponentially}.  In particular, the `quantitative method' sketched above is not sufficient to detect $CE$ for the main example in \cite{Arzhantseva:2011vn}.  We do not know if it is possible to find examples of spaces with $CE$, but without $A$, using the quantitative method above, but leave it as an open problem.
\end{cerem}

\subsection*{Acknowledgments}

Much of the work on this paper was done during visits of the fourth author to the other authors: he would like to thank the Universities of M\"unster and Southampton, and the other authors, for their kind hospitality.  

The fourth author would also like to thank Mikhail Ostrovskii for useful discussions, and especially for bringing the paper \cite{Ostrovskiil:2009kl} to his attention.  This took place during the Concentration Week on Non-Linear Geometry of Banach Spaces, Differentiability and Geometric Group Theory at Texas A\&M University; the fourth author would like to thank the oragnisers of that meeting for creating a very stimulating environment. \\

\noindent
\textsc{Mathematics, University of Southampton, Highfield, Southampton, SO17 1BJ, UK}\\
E-mail: \texttt{j.brodzki@soton.ac.uk}\\

\noindent
\textsc{Mathematics, University of Southampton, Highfield, Southampton, SO17 1BJ, UK}\\
E-mail: \texttt{g.a.niblo@soton.ac.uk}\\

\noindent
\textsc{Mathematisches Institut, Universit\"at M\"unster, Einsteinstr.\ 62, 48149 M\"unster, Germany}\\
E-mail: \texttt{jan.spakula@uni-muenster.de}\\

\noindent
\textsc{2565 McCarthy Mall, University of \Hawaii\ at \Manoa, Honolulu, HI 96822, USA}\\
E-mail: \texttt{rufus.willett@hawaii.edu}\\

\noindent
\textsc{Mathematics, University of Southampton, Highfield, Southampton, SO17 1BJ, UK}\\
E-mail: \texttt{n.j.wright@soton.ac.uk}

\bibliography{Generalbib}

\end{document}